\documentclass[12pt, a4paper]{article}
\usepackage[english]{babel}
\usepackage{amssymb,amsmath,amsthm,mathrsfs}
\usepackage{verbatim}
\usepackage[margin=1in]{geometry}
\usepackage{stmaryrd}

\newtheorem{theorem}{Theorem}
\newtheorem{corollary}[theorem]{Corollary}
\newtheorem{lemma}[theorem]{Lemma}

\newtheorem{definition}[theorem]{Definition}
\newtheorem{proposition}[theorem]{Proposition}

\newtheorem*{example*}{Example}
\newtheorem*{remark*}{Remark}

\def\bal{\begin{aligned}}
\def\eal{\end{aligned}}
\def\be{\begin{equation}\label}
\def\ee{\end{equation}}
\def\bcs{\begin{cases}}
\def\ecs{\end{cases}}
\def\={\;=\;}
\def\+{\,+\,}
\def\-{\,-\,}
\def\C{{\mathbb C}}
\def\Z{{\mathbb Z}}

\def\Q{{\mathbb Q}}
\def\R{{\mathbb R}}

\def\t{\theta}
\def\e{\varepsilon}
\def\ve{\varepsilon}

\def\lb{\llbracket}
\def\rb{\rrbracket}

\def\sD{\mathcal{D}}
\def\sO{\mathcal{O}}
\def\sF{\mathcal{F}}
\def\sH{\mathcal{H}}

\def\sW{\mathcal{W}}
\def\sL{\mathcal{L}}

\def\Im{{\rm Im}}
\def\Hom{{\rm Hom}}

\title{Hodge structures and differential operators}
\author{Masha Vlasenko}

\begin{document}
\date{}
\maketitle

These are extended notes of my talk at the IMPANGA seminar in Warsaw on October 11, 2019. The goal was to give a non-technical and arithmetically motivated introduction to the definition of the limiting mixed Hodge structure by Schmid and Deligne. We state several assertions in terms natural to the classical theory of ordinary differential operators and prove them using elementary arguments. References to the geometric context are only made in a few remarks and examples, which can be easily skipped by a reader not familiar with algebro-geometric techniques.     

In~\S\ref{sec:st-basis} we review a classical method of solving linear differential equations near a regular singular point using Laurent series and the logarithm function. This method produces what we call a standard basis in the space of solutions of the differential equation. The key (and well known) observation is that, if the differential operator itself is a polynomial with coefficients in a field $K \subset \C$, then the power series involved in the standard basis also have coefficients in $K$. 

In \S\ref{sec:PF} we mention geometric (Picard--Fuchs) differential equations and express some period integrals in the standard basis. The upshot of what is done later in \S\ref{sec:limitF} is that the coefficients of such expressions are periods of the limiting Hodge structure. One may wish to skip \S2 on the first reading.

In \S\ref{sec:HS} we review the definition and basic examples of Hodge structures. In \S\ref{sec:limitF} we consider algebraic families of linear functionals on the space of solutions of a polynomial differential operator and define their limits at a singular point. Our main observation (Lemma~\ref{key-lemma} and Corollary~\ref{K-struct-corollary}) is that such limits span the $K$-structure dual to the one determined by the standard basis. We then describe a mixed Hodge structure on the space of solutions, which for geometric differential operators coincides with the limiting mixed Hodge structure constructed by Schmid in~\cite{Sch73}.

\section{Solutions near a regular singular point}\label{sec:st-basis}

Let $K \subset \C$ be a field. A \emph{$K$-structure} on a $\C$-vector space $V$ is a $K$-vector space $V_K \subset V$ such that $V = V_K \otimes_K \C$.

Let $K(t)$ be the field of rational functions with coefficients in $K$. Denote $\t=t \frac{d}{dt}$ and consider a differential operator 
\be{L}
L = \t^r + q_1(t) \t^{r-1} + \ldots + q_{r-1}(t) \t + q_r(t) \; \in \; K(t)[\t]
\ee
whose coefficients satisfy the condition
\be{MUM}
q_j(0) = 0, \qquad 1 \le j \le r.
\ee      
Using classical methods one can construct a set of $r$ linearly independent solutions of~\eqref{L} in the ring $K\lb t \rb [\log(t)]$, see Lemma~\ref{standard-basis-lemma} below. Later we are going to discuss the $K$-structure that this construction yields on the space of solutions of $L$ near $t=0$. 

One starts with finding a Laurent series solution $\phi(t) = \sum a_n t^n$. Expanding the coefficients $q_j(t)$ into power series, let us rewrite $L = \sum_{j \ge 0} t^j p_j(\t)$ with some polynomials $p_j \in K[\theta]$. Condition~\eqref{MUM} is equivalent to $p_0(\t)=\t^r$, and hence the differential equation $L \phi = 0$ is equivalent to the recurrence relation
\[
n^r a_n + p_1(n-1)a_{n-1} + p_2(n-2)a_{n-2}+\ldots = 0.
\] 
Here we see that the smallest $n$ such that $a_n \ne 0$ can only be $n=0$, and with the normalization $\phi(0)=1$ ($a_0=1$) there is a unique Laurent series solution which we will denote $\phi_0(t) = \sum_{n \ge 0} a_n t^n \in 1 + t K\lb t \rb$. Next, one looks for solutions of the shape $\phi(t)=\log(t) \phi_0(t) + \sum_{n}b_n t^n$, and so on:

\begin{lemma}\label{standard-basis-lemma} There exist unique power series $f_0 \in 1 + t \C\lb t \rb$,$f_1 \ldots,f_{r-1} \in t\C\lb t \rb$ such that
\be{standard-basis}
\phi_k(t) = \sum_{j=0}^k \frac{\log(t)^j}{j!} f_{k-j}(t), \qquad k=0,\ldots,r-1,
\ee
are solutions of the differential operator~\eqref{L}. Moreover, one has $f_j \in K\lb t \rb$ for all $0 \le j \le r-1$. 
\end{lemma}

\begin{proof} For $0 \le j \le r-1$ consider the differential operator of order $r-j$ given by $L^{(j)} := \frac{\partial^j L}{\partial \t^j}$, the formal $j$th derivative of $L$ in $\t$. One can easily check that expressions~\eqref{standard-basis} are solutions of $L$ if and only if for each $k$ we have
\be{phi-an-recurrence}
L f_k + L^{(1)} f_{k-1} + \frac1{2!} L^{(2)} f_{k-2} + \ldots \frac1{k!} L^{(k)} f_{0} = 0.
\ee
If we write $L = \sum_{j \ge 0} t^j p_j(\t)$ with $p_j \in K[\t]$ of degree at most $r$ and $p_0(\t)= \t^r$, then equation $L( \sum a_n t^n ) = \sum b_n t^n$ is equivalent to the recurrence relation
\be{recurrence}
n^r a_n + p_1(n-1) a_{n-1} + p_2(n-2) a_{n-2} + \ldots = b_n.
\ee
One can easily see that there is a unique (up to multiplication by a constant) non-zero Laurent series solution $f_0=\sum_{n} a_n t^n$ to $L f_0 = 0$, and this solution is a power series with $a_0 \ne 0$. Moreover, if $a_0 \in K$ then $a_n \in K$ for all $n \ge 0$. Secondly, we observe that $L(K\lb t \rb) \subset t K\lb t \rb$ and, more generally, we have
\[
L^{(j)}(K\lb t \rb) \subset t K\lb t \rb, \qquad 0 \le j \le r-1.
\]
The third observation we can make from formula~\eqref{recurrence} is that the map
\[
L: t \C\lb t \rb \to t \C\lb t \rb
\]
is invertible and maps $t K\lb t \rb$ to itself. With these three observations we can now solve~\eqref{phi-an-recurrence} by induction on $k=0,1,\ldots$ as follows. We start with $k=0$ and normalize the unique power series solution so that $\phi_0=f_0 \in 1 + tK\lb t \rb$. For $k \ge 1$ equation~\eqref{phi-an-recurrence} has shape $L f_k = b$ with $b = - \sum_{j=0}^{k-1} (j!)^{-1} L^{(j)} f_{k-j} \in tK\lb t \rb$. This latter equation has a unique solution $f_k \in t \C\lb t \rb$. Moreover, this $f_k$ has coefficients in $K$, which proves the second assertion of the lemma. 
\end{proof}

\begin{example*} For the differential operator $L=\t^r$ the solutions constructed in Lemma~\ref{standard-basis-lemma} are $\phi_k(t)=(k!)^{-1} \log(t)^k$, $0 \le k \le r-1$.
\end{example*}

\begin{example*} The hypergeometric differential operator $L = \t^2 - \frac{t}{1-t} \t - \frac14 \frac{t}{1-t}$ has solutions
\[\bal
&\phi_0(t) = \, _2F_1(\tfrac12,\tfrac12,1;t) = \sum_{n=0}^\infty \frac{(\tfrac12)_n^2}{n!^2} t^n ,\\
&\phi_1(t) = \, \log(t) \phi_0(t) + \sum_{n=1}^\infty \frac{(\tfrac12)_n^2}{n!^2} \left(\sum_{k=1}^n \frac1{k(k-\tfrac12)}\right)  t^n .\\
\eal\]
\end{example*}

In fact the series $f_j(t)$ in Lemma~\ref{standard-basis-lemma} converge in some neighbourhood of $t=0$. This follows from the classical theorem of Fuchs (see~\cite[Theorem 2.6]{Be-GaussHF}), but one could also estimate the grows of coefficients of $f_j(t)$ directly. Let $V$ be the vector space of solutions to $L$ in a neighbourhood of some regular point $t=t_0$ located close to $t=0$. By Cauchy's theorem $\dim_\C V = r$ and hence functions~\eqref{standard-basis} form a basis in the space of solutions, which we will refer to as \emph{the standard basis}. Note that this basis and the respective $K$-structure
\be{Kstruct}
V_K \= \text{ $K$-span of } \phi_0, \ldots, \phi_{r-1} 
\ee
depend on the choice of branch of $\log(t)$ near the base point $t=t_0$.  

One can write
\[
(\phi_0, \ldots, \phi_{r-1}) = (f_0,\ldots,f_{r-1}) t^{\begin{pmatrix} 0 & 1 & 0 & \ldots \\ 0 & 0 & 1 & \ldots \\ &  &  & \ldots \\ \end{pmatrix}},
\]
where for a square matrix $N$ we use the matrix-valued function $t^N = \exp(N \log(t))=\sum_{h \ge 0} (h!)^{-1}N^h \log(t)^h$. Since $\log(t)$ changes to $\log(t)+2 \pi i$ after a counterclockwise turn around the origin, the respective \emph{local monodromy transformation} $\gamma: V \to V$ in the standard basis is given by the matrix  
\[
\gamma = \exp\left( 2 \pi i \begin{pmatrix} 0 & 1 & 0 & \ldots \\ 0 & 0 & 1 & \ldots \\ &  &  & \ldots \\ \end{pmatrix}\right) = \begin{pmatrix} 1 & 2 \pi i & \tfrac{(2 \pi i)^2}{2!} & \ldots \\ 0 & 1 & 2 \pi i & \ldots \\ &  &  & \ldots \\ \end{pmatrix}.
\]
Observe that this matrix is \emph{maximally unipotent}, that is $(\gamma-I)^r=0$ but $(\gamma-I)^k$ is nonzero when $k < r$. 

\begin{remark*} Condition~\eqref{MUM} means that the differential operator~\eqref{L} has a regular singularity at $t=0$ and the local monodromy around this point is maximally unipotent. More generally, operator~\eqref{L} has at most a regular singularity at $t=0$ if and only if none of the coefficients $q_j(t)$ has a pole at $t=0$. If this is the case, the solutions $\rho \in \C$ to the algebraic equation $\rho^r + \sum_{j=1}^r q_j(0) \rho^{r-j} = 0$ are called local exponents. (The reader could refer to~\cite[\S2]{Be-GaussHF} for the local analysis of singularities of linear differential operators.) By Fuchs' theorem, for each local exponent $\rho$ there is a holomorphic function $f(t)$ with $f(0) \ne 0$ such that $t^\rho f(t)$ is a solution of~\eqref{L}. Moreover, one can construct a basis in the space of solutions of~\eqref{L} near $t=0$ using only functions $\log(t)$, $t^\rho$ for local exponents $\rho$ and holomorphic series. Similarly to Lemma~\ref{standard-basis-lemma}, the power series involved in this \emph{standard basis} will have coefficients in the field $K(\rho_0,\ldots,\rho_{r-1})$. However in this note we will restrict ourselves to the case of maximally unipotent local monodromy.
\end{remark*}

\section{Picard--Fuchs differential operators}\label{sec:PF}

Period integrals are integrals of algebraically defined differential forms over domains described by algebraic equations and inequalities. Period functions are period integrals that depend on a parameter. For example,  the elliptic integral
\be{Legendre-int-1}
t \;\mapsto\; \psi(t) := \int_1^\infty \frac {dx}{\sqrt{x(x-1)(x-t)}} 
\ee
is a period function associated to the Legendre family of elliptic curves
\[
y^2 = x (x-1)(x-t).
\] 
Such functions typically satisfy linear differential equations with algebraic coefficients. For example, expanding the integrand in~\eqref{Legendre-int-1} in a power series in $t$ and integrating term by term one finds that
\be{Legendre-phi0}
\int_1^\infty \frac {dx}{\sqrt{x(x-1)(x-t)}} \= \pi \; _2F_1(\tfrac12,\tfrac12,1; t),
\ee
and hence this period integral is annihilated by the respective hypergeometric differential operator (see the example given in~\S\ref{sec:st-basis}). Vaguely speaking, Picard--Fuchs differential operators are those that annihilate period functions. An excellent introduction to this topic is given in~\cite[Chapter II]{KoZa-Periods}.  

Solution spaces of Picard--Fuchs differential operators have a natural $\Q$-structure, which is preserved by the monodromy transformations. This $\Q$-structure consists of period functions. To be slightly more precise, assume we are given a smooth projective family $f: X \to U$ over an algebraic curve $U = \mathbb{A}^1 \setminus \{\text{roots of } Q\}$ for some polynomial $Q \in K[t]$. Then the relative de Rham cohomology $H^m_{dR}(X/U)$ is a module over the ring of regular functions $\sO_U=K[t,Q(t)^{-1}]$ equipped with a $K$-linear transformation $\tfrac{d}{dt}: H^m_{dR}(X/U) \to H^m_{dR}(X/U)$ (Gauss--Manin connection). Consider the ring  $\sD_U=\sO_U[\tfrac{d}{dt}]$ of differential operators on $U$. We further assume there is a class of differential forms $\omega  \in H^m_{dR}(X/U)$ and a differential operator $L \in \sD_U$ such that $L \omega = 0$. Moreover, we assume that the differential submodule $M = \sD_U \omega \subset H^m_{dR}(X/U)$ is isomorphic to $\sD_U/\sD_U L$. Then the $\Q$-structure on the space $V$ of solutions of $L$ near a regular point  is given by period integrals $V_\Q = \left\{ \int_\gamma \omega  \right\}$ over families of topological cycles $\gamma_t \in H_m(X_t(\C),\Q)$ where $X_t=f^{-1}(t)$ is the fibre at $t$. 

One passes from this $\Q$-structure to the $K$-structure discussed in \S\ref{sec:st-basis} by expressing period integrals in the standard basis~\eqref{standard-basis}:  
\[
\int_\gamma \omega \= \sum_{k=0}^{r-1} \lambda_k \, \phi_k(t).
\]
We will return to this setting at the end of~\S\ref{sec:limitF}, to mention that (under certain assumptions) the coefficients $\lambda_k \in \C$ are periods of the limiting Hodge structure at $t=0$.

For example, the following period integrals for the Legendre family and $\omega =\tfrac{dx}{y}$ can be expressed in the standard basis as
\be{Legendre-int-2}
\bal
&\int_1^\infty \frac {dx}{\sqrt{x(x-1)(x-t)}} \= \pi \, \phi_0(t),\\
&\int_t^1 \frac {dx}{\sqrt{x(x-1)(x-t)}} \;\; \= - 4 i \log(2) \, \phi_0(t) + i \,\phi_1(t).
\eal\ee
(Here we assume that $0 < t < 1$ and the real branch $\log(t) \in \R$ is chosen in $\phi_1(t)$ in the right-hand side.) Since the space of solutions of this Picard--Fuchs differential operator is 2-dimensional, it should be possible to write each of the other two elliptic integrals $\int_{-\infty}^0 \omega$ and $\int_0^t \omega$ as a $\Q$-linear combination of the two integrals given above. One can quickly check that in fact $\int_{-\infty}^0\omega = \int_t^1 \omega$ and $\int_0^t \omega = \int_1^\infty \omega$.
We already mentioned how to obtain the first expression in~\eqref{Legendre-int-2}. The second expression is  trickier. Firstly, we notice that the change of variable $x \mapsto 1 - \frac{1-t}{x}$ transforms the second elliptic integral in~\eqref{Legendre-int-2} into $- i \, \psi(1-t)$ where $\psi(t)$ is the first integral (we introduced this notation in~\eqref{Legendre-int-1}). It remains to show that 
\be{hg-0-1-relation}
\pi \, \phi_0(1-t) = 4 \log(2) \phi_0(t) - \phi_1(t).
\ee 
In order to prove~\eqref{hg-0-1-relation} we shall use the classical relations among hypergeometric series $F(a,b,c; t) := \, _2F _1(a,b,c;t)$ due to Kummer, and also their analytic dependence on the parameters $a,b,c$. 

The following formula is valid for $Re(c)>0, Re(c-a-b)>0$: 
\be{K-rel}\bal
F(a,b,c;1-t) =& \frac{\Gamma(c)\Gamma(c-a-b)}{\Gamma(c-a)\Gamma(c-b)} F(a,b,a+b-c+1;t) \\
&+ \frac{\Gamma(c)\Gamma(a+b-c)}{\Gamma(a)\Gamma(b)} z^{c-a-b} F(c-a,c-b,1+c-a-b; t).
\eal\ee  
We take a small $\e > 0$ and substitute $a=b=\tfrac12-\e$, $c=1-\e$. Observe that the following linear combination yields the log solution in the limit:
\[\bal
& \underset{\e \to 0}\lim \left( \frac{t^{c-a-b} F(c-a,c-b,1+c-a-b; t) - F(a,b,a+b-c+1;t)}{c-a-b}\right)\\
& = \underset{\e \to 0}\lim \; \frac1\e\left(t^\e\sum_{n \ge 0} \frac{(\tfrac12)_n^2}{(1+\e)_n n!}t^n - \sum_{n \ge 0} \frac{(\tfrac12 - \e)_n^2}{(1-\e)_n n!}t^n \right)\\
& = \sum_{n \ge 0} \frac{t^n}{n!} \; \underset{\e \to 0}\lim \; \frac1\e\left( \frac{t^\e (\tfrac12)_n^2}{(1+\e)_n } - \frac{(\tfrac12 - \e)_n^2}{(1-\e)_n }\right) \\
& = \sum_{n \ge 0} \frac{t^n}{n!} \; \underset{\e \to 0}\lim \; \left( \frac{t^\e (\tfrac12)_n^2}{(1+\e)_n }\left(\log(t) -\sum_{k=0}^{n-1} \frac1{1+\e+k}\right) - \frac{(\tfrac12 - \e)_n^2}{(1-\e)_n }\sum_{k=0}^{n-1} \left(-\frac2{\tfrac12-\e+k} + \frac1{1-\e+k}\right)\right) \\
&= \sum_{n \ge 0} \frac{(1/2)_n t^n}{n!^2}\left( \log(t)- \sum_{k=1}^n \frac1{k(k-\tfrac12)}\right) = \phi_1(t).
\eal\]
(The above computation is a known trick, see the part about Kummer relations in~\cite[\S1]{Be-GaussHF}.) Therefore the limit of relation~\eqref{K-rel} when $\e \to 0$ gives
\[\bal
\phi_0(1-t) &= \underset{\e \to 0}\lim \left( \frac{\Gamma(1-\e)\Gamma(\e)}{\Gamma(\tfrac12)^2} + \frac{\Gamma(1-\e)\Gamma(-\e)}{\Gamma(\tfrac12-\e)^2}\right) \phi_0(t) + \underset{\e \to 0}\lim \left( \frac{\e \Gamma(1-\e)\Gamma(-\e)}{\Gamma(\tfrac12-\e)^2}\right)\phi_1(t)\\
&= \frac{4 \log(2)}{\pi} \phi_0(t) - \frac1{\pi} \phi_1(t),
\eal\]
which completes our proof of~\eqref{hg-0-1-relation}.

Note that taking the limit in~\eqref{hg-0-1-relation} as $t \to 1$ one gets the expression
\[\pi = \sum_{n=0}^\infty \frac{(\tfrac12)_n^2}{n!^2}\sum_{k=n+1}^\infty \frac1{k(k-\tfrac12)}.
\]

\section{Hodge structures}\label{sec:HS}

Let $m \in \Z$ be an integer. A \emph{pure Hodge structure of weight $m$} is a $\C$-vector space $V$ with a $\Q$-structure $V_\Q$ and a decreasing filtration $\sF^\bullet V$ satisfying the conditions that for any integers $p,q$ such that $p+q=m+1$ there is a direct sum decomposition $V = \sF^p \oplus \overline{\sF^q}$.

\begin{example*} For a compact complex manifold $\mathfrak{X}$ and $0 \le m \le 2 \dim \mathfrak{X}$, the complexification of the singular cohomology space $V_\Q = H^m(\mathfrak{ X},\Q)$ possesses a pure Hodge structure of weight $m$. This Hodge structure is a basic construction in Hodge theory. 

The following special case is interesting from the arithmetic perspective. When $\mathfrak{X} = X(\C)$ is given by complex points of a smooth projective algebraic variety $X$ defined over a field $K \subset \C$, then there is also a natural $K$-structure given by the de Rham cohomology $V_K = H^m_{dR}(X)$. If $V_\Q$ is the rational structure given by singular cohomology then, in general, $V_K \ne V_\Q \otimes_\Q K$. Comparison of these two structures yields \emph{periods of} $X$.  The Hodge filtration  $\sF^\bullet$ can be defined algebraically, as a filtration of $V_K$. 
\end{example*}

The reader could check the following simple fact as an exercise:

\begin{example*} There is no one dimensional pure Hodge structure of odd weight. For each $k \in \Z$ there is a unique up to isomorphism one-dimensional pure Hodge structure of weight $2k$; it is denoted by $\Q(-k)$.  
\end{example*}

A \emph{mixed Hodge structure} is a $\C$-vector space $V$ with a $\Q$-structure $V_\Q$ and two filtrations: an increasing filtration $\sW_\bullet V$ defined over $\Q$ (that is, it comes from a filtration on $V_\Q$)
and a decreasing filtration $\sF^\bullet V$ satisfying the condition that for every $m$ the graded piece $gr^{\sW}_m V = \sW_m/\sW_{m-1}$ is a pure Hodge structure of weight $m$. Here $\sW_\bullet$ is called the \emph{weight filtration} and $\sF^\bullet$ is called the \emph{Hodge filtration}. 

\bigskip

In 1970's Pierre Deligne showed that the cohomology of a complex variety (possibly singular or non-compact) is a mixed Hodge structure. The following example, in which a mixed Hodge structure arises in a limiting process, will serve as a motivation for our computations in \S\ref{sec:limitF}.

\begin{example*} For a smooth projective family of algebraic varieties $f: X \to U$ over a curve $U \subset \mathbb{A}^1$ one can consider the family of pure Hodge structures of weight $m$ given by $V_t = H^m(X_t)$ where $X_t=f^{-1}(t)$ is the fibre at $t \in U(\C)$. Suppose $0 \in \mathbb{A}^1 \setminus U$ is a singular point and we would like to define the limiting Hodge structure $V$ at $t=0$. We fix a base point $t={t_0}$ close to $t=0$ and define $V_\Q := (V_{t_0})_\Q = H^m(X_{t_0},\Q)$. Since the local monodromy transformation $\gamma : V_{t_0} \to V_{t_0}$ preserves the $\Q$-structure $(V_{t_0})_\Q$, one can think that the $\Q$-structure in the family is not varying. However it happens that $\gamma(\sF^p V_{t_0}) \ne \sF^pV_{t_0}$ for some $p$, which is an obstacle for defining the limiting filtration. In~\cite{Sch73} Schmid introduced a way to remove the monodromy of the Hodge filtration, so that the resulting filtration passes to the limit as $t \to 0$. Let us sketch the approach in~\cite{Sch73}. Take a small punctured disk $\Delta^* = \Delta \setminus \{ 0 \}$ and consider its universal covering by an upper half of the complex plane $e : \sH \to \Delta^*$, $e(z) = \exp(2 \pi i z)$. After one chooses a preimage of the base point $t_0 \in \Delta^*$ in $\sH$ (note that this is equivalent to choosing a branch of $\log(t)$ near $t_0$),  for each $z \in \sH$ there is now a preferred linear isomorphism  $V_{e(z)} \cong V_{t_0}$. This yields a family of filtrations $\sF^\bullet_{z} V$ on $V=V_{t_0}$ indexed by $z \in \sH$. Let $\gamma = \gamma_s \gamma_u$ be the Jordan decomposition of the monodromy transformation into its semisimple and unipotent parts. Consider the nilpotent transformation $N = \log(\gamma_u): V \to V$. Schmid shows (\cite[Theorem 6.16]{Sch73}) there is a limiting filtration
\be{limiting-Fp}
\sF^p_\infty V := \underset{{\rm Im}(z) \to \infty}\lim \exp(- z N) \sF^p_z V,
\ee
and $V$ equipped with $\sF^\bullet_\infty$ and the \emph{monodromy weight filtration} $\sW_\bullet$ is a mixed Hodge structure. The monodromy weight filtration was defined by Deligne as follows. For a nilpotent transformation of a vector space there is an associated filtration  called the Jacobson filtration. If $\sL_\bullet$ is the Jacobson filtration associated to $N: V \to V$, then $\sW_\bullet = \sL_{\bullet - m}$. 
\end{example*}

For later use, we would like to mention that if a nilpotent transformation $N: V \to V$ has one Jordan block and $d = \dim V$ then the respective Jacobson filtration $\sL_\bullet V$ is given by
\be{Jacobson-L}\begin{matrix}
0 & \subset& N^{d-1}(V) & \subset & \ldots & \subset & N(V) & \subset & V \\
\parallel &  & \parallel &  &  &  & \parallel &  & \parallel \\
\sL_{-d} & \subset & \sL_{-d+1}=\sL_{-d+2} & \subset & \ldots & \subset & \sL_{d-3}=\sL_{d-2} & \subset & \sL_{d-1}. \\
\end{matrix}\ee

\section{The limiting Hodge structure on the space of solutions}\label{sec:limitF}

Let us return to the setting of \S\ref{sec:st-basis}. We are given a differential operator    
\[
L = \t^r + q_1(t) \t^{r-1} + \ldots + q_{r-1}(t) \t + q_r(t) \; \in \; K(t)[\t]
\]
with rational coefficients satisfying the condition
\[
q_j(0) = 0, \qquad 1 \le j \le r.
\]
As we explained in \S1, this condition means that $t=0$ is a regular singularity and the local monodromy of solutions of $L$ around this point is maximally unipotent. We shall now give an alternative construction of the $K$-structure on the space of solutions of $L$, which was defined in~\eqref{Kstruct} using the standard basis of solutions.  

We fix a punctured disc $\Delta^* = \Delta \setminus \{ 0 \} = \{ t \in \C \;|\; 0 < |t| < \ve \}$ which is sufficiently small to contain no singularities of $L$. Consider the universal covering 
\[
e: \sH \to \Delta^*, \quad e(z)=\exp(2 \pi i z)
\]
by the respective upper halfplane $\sH = \{ z \in \C \,|\, \Im(z) > - \frac 1{2 \pi} \log(\ve) \}$.   It will be convenient to view multivalued solutions of $L$ in $\Delta^*$ as functions on $\sH$. For an open subset $\mathcal{U} \subset \C$ we denote by $\sO_\mathcal{U}^{an}$ the ring of complex analytic (holomorphic) functions in $\mathcal{U}$. Consider the solution space 
\[
V := \{ u(z) \in \sO^{an}_{\sH} \,|\, (e^*L)u=0\},
\]
where $e^*L = (\tfrac 1{2 \pi i} \tfrac{d}{dz})^r + \sum_{j=0}^r q_j(e(z)) (\tfrac 1{2 \pi i} \tfrac{d}{dz})^{r-j}$ is the pullback of $L$ to $\sH$. This differential operator is obtained via the substitution $t=e(z)$ in $L$. By Cauchy's theorem $\dim_\C V = r$. The basis in $V$ given by
\be{st-basis-z}
\phi_k(z) = \sum_{j=0}^k \frac{(2 \pi i z)^j}{j!} f_{k-j}(e(z)) , \qquad 0 \le k \le r-1 
\ee
with $f_j \in \sO_\Delta^{an}$ defined in Lemma~\ref{standard-basis-lemma} will be called the \emph{standard basis}. The monodromy transformation $\gamma : V \to V$ acts by 
\be{M0}
(\gamma u)(z)=u(z+1).
\ee 
The space of linear functionals on the solution space will be denoted by $V^\vee := Hom_\C(V,\C)$.

\begin{definition}\label{alg-family-def} A family of linear functionals $\pi_z \in V^\vee$ indexed by $z \in \sH$ is an \emph{algebraic family} if it is given by
\be{alg-functionals}
\pi_z : u \mapsto \sum_{j=0}^{r-1} \frac{v_j(e(z)) }{ (2 \pi i)^{j}} \left( \frac{d^ju}{dz^j} \right)(z)
\ee
for some rational functions $v_0,\ldots,v_{r-1} \in K(t)$ having no poles in $\Delta^*$. To such a family we associate its \emph{symbol} $m = \sum_{j=0}^{r-1} v_j(t) \t^j$, and we will write $\pi_z=\pi_z(m)$ to denote the respective linear functionals~\eqref{alg-functionals}. 

An algebraic family of linear functionals with symbol $m = \sum_{j=0}^{r-1} v_j(t) \t^j$ is said to be \emph{analytic at $t=0$} if none of the coefficients $v_j \in K(t)$ has a pole at $t=0$.  
\end{definition}

We would like to consider the limits of $\pi_{z} \in V^\vee$ as $\Im(z) \to +\infty$ in algebraic families. However one can easily check in formula~\eqref{alg-functionals} that $\pi_{z+1} \ne \pi_{z}$, which is an obvious obstruction for taking the above mentioned limit. We will now remove this monodromy using Schmid's formula~\eqref{limiting-Fp} (see~\cite[(6.15)]{Sch73}):

\begin{lemma}\label{key-lemma} Let $N = \log(\gamma)=\sum_{h \ge 1} (-1)^{h-1} (\gamma-I)^h/h$ be the logarithm  of the local monodromy transformation $\gamma$. For an algebraic family of linear functionals $\pi_z=\pi_z(m)$ corresponding to $m = \sum_{j=0}^{r-1} v_j(t) \t^j$ the family of linear functionals defined by 
\[
\pi'_{z} := \exp(- z N) \circ \pi_{z}  \in V^\vee
\]
satisfies $\pi'_{z+1} = \pi'_{z}$. The limit $\lim_{\Im(z) \to +\infty}  \pi'_{z}$ exists whenever $(\pi_z)$ is analytic at $t=0$, in which case one has
\be{lim-m}
\lim_{\Im(z_0) \to +\infty} \pi'_{z} \=  \sum_{j=0}^{r-1} v_j(0) \phi_j^\vee. 
\ee
Here  $\phi_0^\vee, \phi_1^\vee, \ldots$ is the basis in $V^\vee$ dual to the standard basis~\eqref{st-basis-z}.
\end{lemma}

\begin{proof} Note that every solution $u \in V$ can be uniquely written as
\[
u(z) = \sum_{k=0}^{r-1} u_k(e(z)) z^k
\]
with $u_0(t),\ldots,u_{r-1}(t) \in \sO_\Delta^{an}$. The action of $N$ on $V$ is given by
\[
N : \sum_k u_k(e(z)) z^k \mapsto \sum_k u_k(e(z)) k \, z^{k-1}.
\]
For $m = \sum_j v_j \t^j$ we evaluate $\pi_z' = \exp(- z N) \circ \pi_{z}(m)$ on $u = \sum_k u_k(e(z)) z^k \in V$ as follows:
\[\bal
\pi'_{z}(u) &\= \sum_{h \ge 0} \frac{(-z)^h}{h!} \pi_z(m)(N^h u) \\
&\= \sum_{h \ge 0} \sum_{j,k=0}^{r-1} \frac{(-z)^h}{h!} \frac{v_j(e(z))}{(2 \pi i)^{j}} \left(\frac{d}{dz}\right)^j \left( u_k(e(z))   \left(\frac{d}{dz}\right)^h z^k \right) \\
&\= \sum_{h \ge 0} \sum_{j,k=0}^{r-1} \frac{(-z)^h}{h!} v_j(e(z)) \sum_{s=0}^j (\t^{j-s} u_k)(e(z))  \frac{k(k-1)\ldots(k-h-s+1)}{(2 \pi i)^{s}} z^{k-h-s} \\
&\= \sum_{j,k=0}^{r-1} \sum_{s=0}^{\min(j,k)} v_j(e(z)) (\t^{j-s} u_k)(e(z))(2 \pi i)^{-s} \frac{k!}{(k-s)!} z^{k-s}\sum_{h \ge 0} (-1)^h \binom{k-s}{h} \\
&\= \sum_{j \ge k} v_j(e(z)) (\t^{j-k} u_k)(e(z)) \frac{k!}{(2 \pi i)^{k}}.\\
\eal\]
Note that the function in the last row is periodic in $z$, so we conclude that $\pi'_{z+1}=\pi'_{z}$. When all $v_j(t)$ are analytic at $t=0$ the above expression passes to the limit as $\Im(z)$ grows infinitely:
\[
\lim_{\Im(z) \to +\infty} \pi'_{z}(m)(u) \=  \sum_{j=0}^{r-1} v_j(0) u_j(0) \frac{j!}{(2 \pi i)^j}.
\] 
It remains to notice that the linear functional mapping $u=\sum_k u_k(e(z)) z^k$ into $u_j(0) (2 \pi i)^{-j}\,j!$
coincides with $\phi_j^\vee$. This fact follows from formula~\eqref{st-basis-z} because $f_0(0)=1$ and $f_j(0)=0$ when $j>0$.
\end{proof}

The following observation is a key fact relating the $K$-structure described in \S\ref{sec:st-basis} with the limiting process defined in Lemma~\ref{key-lemma}:  

\begin{corollary}\label{K-struct-corollary}
The vector space
\[
V^\vee_K := \text{$K$-span of }\; \phi_0^\vee, \ldots, \phi_{r-1}^\vee
\]
consists of the limits at $t=0$ of algebraic families of linear functionals on the solution space $V$. 
\end{corollary}
\begin{proof} Since in formula~\eqref{lim-m} all $v_j(0) \in K$,  we obtain $\lim_{\Im(z) \to \infty} \pi'_z \in V_K^\vee$. Conversely, any functional $\sum \lambda_j \phi_j^\vee \in V_K^\vee$ is the limit of the family $\pi_z'(m)$ corresponding to the symbol $m = \sum_j \lambda_j \t^j$.\end{proof}

Our next goal is to describe a mixed Hodge structure on $V^\vee$. Recall that $N = \log(\gamma)$ is a nilpotent transformation, and hence images of its powers define a finite filtration on $V^\vee$. We define the weight filtration $\sW_\bullet V^\vee$ as follows:
\be{W}\begin{matrix}
0 & \subset& N^{r-1}(V^\vee) & \subset & \ldots & \subset & N(V^\vee) & \subset & V^\vee \\
\parallel &  & \parallel &  &  &  & \parallel &  & \parallel \\
\sW_{-1} & \subset & \sW_{0}=\sW_{1} & \subset & \ldots & \subset & \sW_{2r-4}=\sW_{2r-3} & \subset & \sW_{2r-2}. \\
\end{matrix}\ee
This filtration is the shift $\sW_\bullet = \sL_{\bullet-r+1}$ of the Jacobson filtration associated to $N$ (see~\eqref{Jacobson-L}). Note that it is is naturally defined on any $\Q$-structure preserved by the local monodromy transformation $\gamma$.

To define the Hodge filtration we will identify elements of $V^\vee$ with the limits of families of linear functionals at $t=0$ and filter families by the order of their symbol in $\t$. Consider the ring of differential operators $\sD = K(t)[\t]$ and note that symbols $m=\sum_{j=0}^{r-1} v_j(t) \t^j$ of algebraic families of linear functionals in Definition~\ref{alg-family-def} can be thought as elements of the differential module $M = \sD/\sD L \cong \sum_{j=0}^{r-1} K(t) \t^j$. More generally, we would like to consider families with analytic coefficients $v_j \in \sO_\Delta^{an}$. Their symbols are elements of  
\[
\widetilde M \= \widetilde\sD/\widetilde\sD L \;\cong\; \sum_{j=0}^{r-1} \sO_\Delta^{an} \; \t^j,
\]
where $\widetilde\sD=\sO_\Delta^{an}[\t]$ is the ring of differential operators whose coefficients are holomorphic functions in $\Delta$. For each $z \in \sH$ formula~\eqref{alg-functionals} yields a map 
\[
\pi_z: \widetilde M \to V^\vee.
\]
We now consider a decreasing filtration of $\widetilde M$ by $\sO_\Delta^{an}$-modules given by 
\[
\sF^p \widetilde M := \sum_{j=0}^{r-1-p} \sO^{an}_\Delta \, \t^j, \quad 0 \le p \le r-1
\]
and define a family of filtrations on $V^\vee$ indexed by $z \in \sH$ via
\[
\sF^p_z V^\vee \,:=\, \pi_z \left( \sF^p \widetilde M \right) , \quad 0 \le p \le r-1.
\]
The computation in Lemma~\ref{key-lemma} applies also to analytic families. An immediate consequence is the following

\begin{corollary}\label{Flim} The limiting filtration
\[
\sF^p_\infty V^\vee := \underset{\Im(z) \to \infty}\lim \exp(- z N) \sF^p_z V^\vee, \quad 0 \le p \le r-1
\]
exists and is given by 
\[
\sF^p_\infty V^\vee \= \C\text{-span of } \phi_0^\vee, \ldots, \phi_{r-1-p}^\vee.
\]
\end{corollary}

\begin{proposition}\label{MHS-cor} Let $\sW_\bullet$ be the Jacobson filtration associated to the nilpotent transformation $N = \log(\gamma)$ of $V^\vee$ and shifted by $r-1$. Let $\sF^\bullet_\infty$ 
be the limiting filtration from Corollary~\ref{Flim}. Then for any $\Q$-structure on the solution space $V_\Q$ preserved by the monodromy transformation $\gamma$ the triple $(V_\Q^\vee, \sW_\bullet, \sF^\bullet_{\infty})$ is a mixed Hodge structure.
\end{proposition}

\begin{proof} Filtration $\sW_\bullet$ is given explicitly by~\eqref{W}. Note that $N^j(V^\vee)= Span_\C(\phi_j^\vee,\ldots,\phi_{r-1}^\vee)$, and hence the two filtrations are opposite in the sense that $V^\vee = \sW_{2k} \oplus \sF^{k+1}_{\infty}$ for each $0 \le k \le r-1$. It follows that the filtration induced by $\sF^\bullet_{\infty}$ on $gr_{2k}^\sW \, V^\vee$ is zero in degrees $> k$ and everything in degree $k$.
\end{proof}

Note that the pure graded pieces of the mixed Hodge structure in Proposition~\ref{MHS-cor} are one dimensional. With the notation explained in \S\ref{sec:HS} we have 
\[gr^{\sW} V^\vee \cong \oplus_{j=0}^{r-1}\Q(-k).
\]

Let us also mention that one can always find a $\Q$-structure preserved by $\gamma$. For example, the reader may check that $\gamma$ preserves the $\Q$-span of $(2\pi i)^{-k} \phi_k$, $0 \le k \le r-1$. For a deeper example one can turn to Picard--Fuchs differential operators. As we mentioned in \S\ref{sec:PF}, their spaces of solutions possess a natural $\Q$-structure consisting of period functions.

\begin{remark*} The construction of this section applies to Picard--Fuchs differential operators with maximally unipotent local monodromy. Namely, suppose $X \to U$ is a smooth projective family and in the setting described in \S\ref{sec:PF} we assume in addition that $\omega$ belongs to the smallest $\sO_U$-submodule in the Hodge filtration $\sF^m H^{m}_{dR}(X/U)$, and that the operator $L \in \sD_U$ annihilating $\omega$ is of order $r=m+1$. Then using the Griffiths' transversality property of a variation of Hodge structure and the fact that the monodromy of $L$ is maximally unipotent, one can show that the Hodge filtration on $M=\sD_U / \sD_U L \cong \sD_U \omega \subset H^{r-1}_{dR}(X/U)$  is given by~$\sF^{\bullet} M = \sum_{j=0}^{r-1-\bullet} \sO_U \t^j \omega$. Solutions of $L$ can be identified with horizontal sections of the analytification $M^\vee_{an} = M^\vee \otimes_\sO \sO^{an}$ of the dual connection $ M^\vee = \Hom_{\sO}(M,\sO)$. Over the punctured disk $\Delta^*$, this identification yields the pairing of elements $m = \sum_j v_j \t^j \in M$ with solutions $u \in V$. The result of this pairing is the analytic function $z \mapsto \pi_z(u)$ given by the right-hand side in our formula~\eqref{alg-family-def}. Since the limiting process of Lemma~\ref{key-lemma} coincides with the one in~\cite{Sch73}, the mixed Hodge structure in Proposition~\ref{MHS-cor} is the limiting mixed Hodge structure of Deligne and Schmid. 

In this geometric situation, there is a fine notion of the de Rham structure of the limiting mixed Hodge structure. The arguments given in~\cite[Remark 42]{BlVl} indicate that it should coincide with our $K$-structure $V_K$ spanned by the standard basis in the space of solutions. 
\end{remark*}

\bigskip
\bigskip

\noindent{\bf Afterword.} The content of this note is essentially covered by Lemma 41 in~\cite{BlVl}, where we use the ideas explained here to compute periods of limiting Hodge structures. I am grateful to Spencer Bloch for introducing me to this interesting subject.  Special thanks to Wadim Zudilin, whose expertise in hypergeometric functions helped to work out the example given in \S\ref{sec:PF}.

\end{document}